\documentclass[12pt]{amsart}
\usepackage{amsthm,amssymb,amsmath,amscd,epic,eepic}
\usepackage[all]{xy}

\newcommand{\mute}[2] {}
\newcommand     {\printname}[1] {}
\newcommand{\labell}[1] {\label{#1}}

\numberwithin{equation}{section}
\newtheorem {Theorem}[equation]                   {Theorem}
\newtheorem*{Theorem*}                   {Theorem}

\newtheorem {Lemma}[equation]           {Lemma}

\newtheorem {Corollary} [equation]      {Corollary}
\newtheorem* {Corollary*}                {Corollary}
\newtheorem {Proposition} [equation]    {Proposition}
\newtheorem {Question} [equation]    {Question}

\newtheorem* {Lemma*}                    {Lemma}

\theoremstyle{definition}
\newtheorem{Definition}[equation]{Definition}
\newtheorem*{Definition*}{Definition}

\theoremstyle{remark}
\newtheorem{Remark}[equation]{Remark}
\newtheorem*{Remark*}{Remark}

\long\def\symbolfootnote[#1]#2{\begingroup%
\def\thefootnote{\fnsymbol{footnote}}\footnote[#1]{#2}\endgroup}

\def    \CP	{{\mathbb C}{\mathbb P}}
\def    \R	{{\mathbb R}}
\def    \N	{{\mathbb N}}

\def    \Z       {{\mathbb Z}}

\begin{document}

\title[Eight dimensional Hamiltonian $S^1$-spaces with 5 fixed points]
{Hamiltonian circle actions on eight dimensional manifolds with minimal fixed sets}

\author[Donghoon Jang]{Donghoon Jang}
\address{Department of Mathematics, University of Illinois at Urbana-Champaign,
 Urbana, IL 61801}
\address{School of Mathematics, Korea Institute for Advanced Study, 85 Heogiro, Dongdaemun-gu, Seoul, Korea}
 \email{groupaction@kias.re.kr, jang12@illinois.edu}

\author[Susan Tolman]{Susan Tolman}
\address{Department of Mathematics, University of Illinois at Urbana-Champaign,
 Urbana, IL 61801}
\email{stolman@math.uiuc.edu}


\thanks{Susan Tolman is partially supported by NSF Grant DMS \#12-06365.}
\thanks{Donghoon Jang is partially supported by Campus Research Board Awards - University of Illinois at Urbana-Champaign.}

\begin{abstract}
Let the  circle act  in a Hamiltonian fashion 
on a closed  $8$-dimensional  symplectic manifold $M$ with exactly five fixed points,
which is the smallest possible fixed set.
In  \cite{GS}, L. Godinho and S. Sabatini show that if $M$ satisfies an
extra ``positivity condition'' then the isotropy weights at the fixed points of $M$
agree with those of some linear action on $\CP^4$.  
As a consequence,
$H^*(M;\Z) = \Z[y]/y^5$ and $c(TM) = (1+y)^5$.
In this paper, we prove that their positivity condition  holds for 
$M$.
This completes the proof of the ``symplectic Petrie conjecture" for
Hamiltonian circle actions on 8-dimensional closed symplectic manifolds with minimal fixed sets.
\end{abstract}

\maketitle
\section{Introduction}

The classification of Hamiltonian $S^1$-actions is an important subject in symplectic 
geometry. Although in general the problem is probably too hard to be tractable,
progress has been made in special cases.
In particular, 
let $(M,\omega)$ be a closed $2n$-dimensional symplectic manifold with a Hamiltonian 
circle action
such that $H^{2i}(M;\R) \cong H^{2i}(\CP^n;\R)$ for all $i$.
Since $[\omega]^i \neq 0$ for $0 \leq i \leq n$, this means that $M$ has the smallest possible even Betti numbers.
In analogy with the Petrie conjecture \cite{P}, one can reasonably hope to classify all 
the
possible (equivariant) cohomology rings and (equivariant) Chern classes of the manifold $M$.

This project, which was first proposed by the second author in \cite{T},
 has been successfully carried out in a number of special cases.
In that paper, the author shows that if $\dim M = 6$  only four cohomology rings can arise, 
and the Chern classes are uniquely determined by the ring in each case. 
(An earlier paper also solved  the problem when  $\dim M = 6$  and the fixed points are isolated \cite{A}; see also \cite{H2}.)
Additionally,  Hattori's results in \cite{H1} imply that if $\dim M = 2n$ the first Chern class
of $M$ must be at most $n+1$ times a primitive generator of $H^2(M;\Z)$.  
Moreover, if the multiple is $n+1$, the cohomology ring and Chern classes of $M$ agree with those of $\CP^n$;
if the multiple is $n$ then $n > 1$ is odd and they agree with those of
${G}^+_{2}(\R^{n+2})$, the Grassmannian of oriented two planes
in $\R^{n+2}$.  
Alternatively, if $M$ has only two fixed components then again the cohomology ring and Chern classes of $M$ agree
with those of either $\CP^n$ or  ${G}^+_{2}(\R^{n+2})$ \cite{LT}.
If $M$ is K\"ahler, the same claim holds in some cases with three fixed components \cite{Lh1}.
Finally, several papers consider the same situation but work with finer geometric or topological invariants,
classifying such spaces up to (equivariant) diffeomorphism or symplectomorphism \cite{M, LOS, Lh1}.

In this paper, we consider the case that all the fixed points are isolated.
Because the family of almost complex structures $J$ on $M$ compatible
with $\omega$ is contractible, this implies that there
is a multiset of $n$ non-zero integers $w_{p1},\dots,w_{pn}$, called {\bf isotropy weights},
associated to each fixed point $p \in M^{S^1}$.
Moreover, the moment map $\phi \colon M \to \R$ is a perfect Morse function, 
the critical points are the fixed points, and the index of each critical point is twice the 
number of negative weights (counted with multiplicity). 
Therefore, there are at least $n+1$ fixed points, and 
$H^{2i}(M;\R) \cong H^{2i}(\CP^n;\R)$ for all $i$
if and only if there are exactly $n+1$ fixed points --
one fixed point of index $2i$ for each $i \in \{0,\dots,n\}$.
In this case, the following holds \cite{T}: {\it The isotropy weights at the fixed points completely determine the ordinary cohomology ring $H^*(M;\Z)$,  
the equivariant cohomology ring $H^*_{S^1}(M;\Z)$, 
its restriction to the equivariant cohomology of the fixed set $H^*_{S^1}(M^{S^1};\Z)$,
and the (equivariant) Chern classes}\footnote{
The {\bf equivariant cohomology} of $M$ is $H_{S^1}^*(M) := H^*(M \times_{S^1} E S^1)$, where $ES^1$ is a contractible space on which $S^1$ acts freely.
Similarly,  the {equivariant Chern classes} of $M$ are the Chern classes
of $T(M) \times_{S^1} E S^1 \rightarrow M \times_{S^1} E S^1$.}.
Hence, to classify the possible cohomology rings and Chern classes,
it  is enough to classify the possible isotropy weights.

In their paper \cite{GS}, L. Godinho and S. Sabatini show that if $M$ is $8$-dimensional, has exactly fixed fixed points,  and
satisfies an 
extra ``positivity condition'' $(\mathcal P^+_0)$, then
the isotropy weights at the fixed points of $M$ agree with those of some linear action on $\CP^4$;
see Definitions~\ref{multigraph}, \ref{describes}, and  \ref{pos},  and
Theorem~\ref{GStheorem}.
In Theorem~\ref{positive}, we show that the positivity condition $(\mathcal P^+_0)$ is always satisfied for these manifolds. 
Thus, we prove the following theorem:

\begin{Theorem}\labell{main}
Consider a Hamiltonian circle action on a
$8$-dimensional closed symplectic  manifold $(M,\omega)$ with exactly  $5$ fixed points.
Then the isotropy weights at the fixed points of $M$ agree with those of some linear action on 
$\CP^4$.
As a consequence,  $H^* (M;\mathbb{Z})=\mathbb{Z}[y]/y^5$ and $c(TM)=(1+y)^5$;
moreover,  the equivariant cohomology and equivariant Chern classes of
$M$ and $\CP^4$ agree.
\end{Theorem}

\begin{Remark}
If $M$ is a homotopy $\mathbb{CP}^4$, then by \cite{J} the (equivariant) cohomology and (equivariant) Pontryagin classes of $M$ and $\CP^4$ agree.
\end{Remark}

To prove Theorem~\ref{positive}, we  develop several tools  
to analyze the labelled multigraphs that describe circle actions
on almost complex manifolds with isolated fixed points\footnote{
See Definitions~\ref{multigraph} and \ref{describes}.};
see Lemma~\ref{decreases} and Corollary~\ref{cor:decreases}.
We believe that these will be useful more generally in classifying such actions.
For example, these claims could have reduced the number of cases that Godinho and Sabatini needed to consider in order to prove Theorem~\ref{GStheorem}; cf. \cite{GS}. 

Unfortunately, the proof of Theorem~\ref{positive} fails in higher dimensions;
we do not know if all Hamiltonian circle actions with minimal fixed sets satisfy the positivity condition $(\mathcal P_0^+)$.
However, condition 
$(\mathcal P^+_0)$ 
is  not satisfied for general Hamiltonian actions; see  \cite[Remark 5.7]{GS}.
In contrast, we do not know any counter-examples to the following general question; an affirmative answer to it would imply that $(\mathcal P^+_0)$ holds
whenever the fixed set is minimal:

\begin{Question}
Let the circle act on a closed symplectic manifold $(M,\omega)$ with isolated fixed points and with moment map $\phi \colon M \to R$. 
Does there exist a labelled multigraph describing
$M$ so that 
$\phi(i(e)) < \phi(t(e))$
for every edge $e$?
\end{Question}

Finally, we would like to thank the  anonymous referees for helping us to improve our exposition.

\section{The main technical tool}

In this section, we prove our main technical tool, which gives  restrictions on the smallest positive weight for actions
on almost complex manifolds (and therefore symplectic manifolds).

Let the circle act on a closed $2n$-dimensional almost complex manifold with isolated fixed points\footnote{
We assume the action preserves the almost complex structure.}.
As before, there is a multiset $w_{p1},\dots,w_{pn}$ of {\bf isotropy weights} at each fixed point $p$.
In analogy with the Hamiltonian case, define the {\bf index} of a fixed point $p$ to be $2 \lambda_p$,
where  $\lambda_p$ is the number of negative isotropy weights at $p$.
Let $N^k$ be the number of fixed points of index $2k$, and let $\sigma_k$ denote the elementary symmetric polynomial of degree $k$ in $n$ variables.
In \cite{Lp}, P. Li proves the following:

\begin{Theorem}\label{Li} 
Let the circle act  on a closed $2n$-dimensional almost complex  manifold $(M,J)$ with isolated fixed points. 
Then for all $0 \leq k \leq n$,
$$
\sum_{p \in M^{S^1}} \frac{\sigma_k (t^{w_{p1}}, \cdots, t^{w_{pn}})}{\prod_{j=1}^n (1-t^{w_{pj}})}=(-1)^k N^k.$$
\end{Theorem}

Kosniowski proved a claim analogous to Proposition~\ref{match} below
for holomorphic vector fields on complex manifolds with simple isolated zeroes \cite{K}.
Closely following his idea, we use Theorem~\ref{Li} to prove the following:

\begin{Proposition}\labell{match} Let the circle act  on a closed  $2n$-dimensional almost complex manifold $(M,J)$ with 
isolated fixed points.  
Let $a$ be the smallest positive isotropy weight that occurs at any fixed point.
Given any $k \in \{0,1,\dots,n-1\}$,
the number of times the isotropy weight $-a$ occurs at fixed points of index $2k +2 $ is equal to
the number of times the isotropy weight $+a$ occurs at fixed points of index $2k$.
\end{Proposition}

\begin{proof}
Fix an integer $k \in \{0, 1,\dots,n- 1\}$.
By Theorem~\ref{Li},
\begin{multline*}\label{lieq}
N^k = \sum_{p \in M^{S^1}} (-1)^k \frac{  \sigma_k (t^{w_{p1}}, \cdots, t^{w_{pn}})}{\prod_{j=1}^n (1-t^{w_{pj}})} \\
=\sum_{p \in M^{S^1}}  (-1)^{\lambda_{p} + k} \frac 
{
 \sigma_k (t^{w_{p1}}, \cdots, t^{w_{pn}})
 \prod_{w_{pj}<0} t^{-w_{pj}}
}
{\prod_{j=1}^n (1-t^{|w_{pj}|})} .
\end{multline*}
Here, $\prod_{w_{pj} < 0} t^{- w_{pj}}$ denotes the product of the $t^{- w_{pj}}$ over 
$j \in \{1,\dots,n\}$ satisfying $w_{pj} < 0$.

Moreover, 
by reversing the action if necessary, we may assume that there is no  isotropy weight $b$  at any fixed point with $|b| < a$.
Thus, as $t \to 0$, at any fixed point $p$ we have
$$
\sigma_k (t^{w_{p1}}, \cdots, t^{w_{pn}})
\textstyle \prod_{w_{pj}<0} t^{-w_{pj}} = 
\begin{cases}
\nu_{p}(a)t^a + O(t^{a+1}) & \mbox{if } \lambda_{p} = k - 1 \\
1+O(t^{2a }) & \mbox{if }  \lambda_{p}=k \\
\nu_{p}(-a)t^a + O(t^{ a + 1}) & \mbox{if }  \lambda_{p} = k + 1 \\
O(t^{2a }) & \mbox{otherwise,}
\\
\end{cases}
$$
where $\nu_{p}(\pm a)$ is the number of times the weight $\pm a$ occurs at $p$.
Similarly, 
$$ \prod_{j=1}^n \left(1 - t^{|w_{pj}|}\right) = 1 -  \big( \nu_{p}(a) + \nu_{p}(-a)\big) t^a + O(t^{a+1}).$$

Together, the three displayed equations imply that
\begin{multline*}
N_k =
- \sum_{ \lambda_{p}=k-1} \nu_{p}(a) t^a +
\sum_{  \lambda_{p}=k} \big(1 -  \big(\nu_{p}(-a)+\nu_{p}(a)\big)t^a\big)  \\
- \sum_{ \lambda_{p}=k+1} \nu_{p}(-a)t^a  
+ O(t^{a+1}),  \\
\end{multline*}
where each sum is over all fixed points with the given number of negative weights. 
Therefore, since $N_k = \sum_{\lambda_p = k} 1$, taking the leading  coefficient of the Taylor expansion gives us
\begin{equation}\labell{key}
\sum_{  \lambda_{p}=k} \big(\nu_{p}(-a)+\nu_{p}(a)\big) =\sum_{ \lambda_{p}=k-1} \nu_{p}(a)+\sum_{  \lambda_{p}=k+1} \nu_{p}(-a). 
\end{equation}

By definition, there are no fixed points with negative index, and points of index $0$ have
no negative weights. Hence for $k = 0$ \eqref{key} simplifies to
\begin{equation*}
\sum_{  \lambda_{p}=0} \nu_{p}(a) =\sum_{\lambda_{p}=1} \nu_{p}(-a),
\end{equation*}
which proves the claim in this case.
Moreover, if we assume that $\sum_{\lambda_p = k - 1} \nu_p(a) =  \sum_{\lambda_p = k} \nu_p(-a)$,
then \eqref{key} immediately implies that 
$\sum_{\lambda_p = k } \nu_p(a) =  \sum_{\lambda_p = k + 1} \nu_p(-a)$.
Hence, the claim follows by induction.
\end{proof}

\section{Consequences for multigraphs}

The main goal of this section is to prove Theorem~\ref{positive}.
First, we review  some material, 
which we adapt from \cite{GS}.
Let the circle act on a closed almost-complex manifold $(M,J)$ 
with isolated fixed points.
Given a fixed point $p \in M^{S^1}$ with isotropy weights $w_{p1},\dots,w_{pn}$,
let
$$\Gamma_p = \sum_{i=1}^n w_{pi}.$$

\begin{Definition}\labell{multigraph}
A {\bf labelled (directed) multigraph} is a set $V$ of vertices, a set $E$ of edges, 
maps $i \colon E \to V$ and $t \colon E \to V$
giving the initial and terminal vertices of each edge, and a map $w$ from $E$ to the positive integers. 
\end{Definition}

We say that a  multigraph contains a  {\bf self-loop} if there exists an
edge $e$ that connects a vertex to itself, that is, $i(e) = t(e)$.
Similarly, the multigraph  has {\bf multiple edges} if there exist distinct edges $e$ and $e'$
with $i(e) = i(e')$ and $t(e) = t(e')$.  A priori, both are allowed.

The next two definitions
correspond to \cite[Definition 4.9]{GS} and \cite[Definition 4.12]{GS}, respectively.

\begin{Definition}\labell{describes}
We say that a labelled  multigraph with vertex set $M^{S^1}$ {\bf describes} $M$ if the following hold:
\begin{enumerate}
\item  The multiset of isotropy weights at  $p$ is
 $\{ w(e) \mid i(e) = p \} \cup 
\{ - w(e) \mid t(e) = p \}$ for all $p \in M^{S^1}$; and
\item For each edge $e$, the two endpoints $i(e)$ and $t(e)$ are in the same component of the isotropy submanifold $M^{\Z/(w(e))}$.
\end{enumerate}
\end{Definition}

\begin{Definition}\label{pos}
We say that $M$
satisfies the {\bf positivity property} $(\mathcal P^+_0)$  if
it can be described by a labelled multigraph so that
$$\Gamma_{i(e)} \geq \Gamma_{t(e)} $$
for every edge $e$.
\end{Definition}

The following theorem is an immediate consequence of \cite[Theorem  1.3]{GS}.

\begin{Theorem}\label{GStheorem}
Consider a Hamiltonian circle action on an $8$-dimensional closed symplectic manifold 
$M$ with exactly five fixed points. If 
$M$ satisfies the positivity property $(\mathcal P^+_0)$, then 
the isotropy weights at the fixed points of $M$ agree with those of some linear action on
$\CP^4$.
\end{Theorem}

Here, a {\bf linear action} on $\CP^4$ is the action induced by  an embedding 
$S^1  \hookrightarrow U(5)$, where
$U(5)$ acts on  
$\CP^4$ by the defining representation.  

With this theorem as motivation, we explore the consequences of Proposition~\ref{match}
on the labelled multigraphs that describe actions on almost complex manifolds.
Our first lemma shows that we do not need multigraphs with self-loops.

\begin{Lemma}\labell{decreases}
Let the circle act  on a closed almost-complex manifold $(M,J)$ with isolated fixed points.
There exists a labelled multigraph describing $M$ such that, for each edge $e$,
the index of $i(e)$ in the isotropy submanifold $M^{\Z/(w(e))}$ is two less than the
index of $t(e)$ in $M^{\Z/(w(e))}$. 
In particular, the multigraph has no self-loops.
\end{Lemma}

\begin{proof}
Fix $l \in \N$. There cannot be any positive weight in  the isotropy submanifold $M^{\Z/(l)}$ that is smaller than $l$.
Therefore, by applying Proposition~\ref{match} to each component $Z \subset M^{\Z/(l)}$, 
we see that
the number of times the weight $-l$ occurs in points of index $2 k$ in $Z$ is
equal to the number of times the weight $+l$ occurs in points of index $2 k-2$ in $Z$.
Both claims follow immediately.

\end{proof}

\begin{Remark}
In contrast, we sometimes  need multigraphs with multiple edges.
For example, there are circle actions on  $S^6 \cong G_2/SU(3)$
that preserve a natural almost complex structure   and have exactly two fixed points.
Any labelled multigraphs describing these examples must have multiple edges.
Similarly, for $\ell = 4$ or $5$,
there exists a  Hamiltonian circle action on a $6$-dimensional closed 
symplectic manifold with exactly four  fixed points that have
the following isotropy weights: $\{-1,-2,-3\}, \{1,-1,-\ell\}, \{1,\ell,-1\},$ and $\{1,2,3\}$ \cite{M}.
Again, any labelled multigraphs describing this example must have multiple edges.
\end{Remark}

This lemma has the following consequence for Hamiltonian actions.

\begin{Corollary}\labell{cor:decreases}
Let the circle act  on a $2n$-dimensional closed symplectic manifold $(M,\omega)$ with isolated fixed points and with moment map  $\phi \colon M \to \R$.
There exists a multigraph describing $M$ with the following property:
if $e$ is an edge such that either the index of $t(e)$ is $2$ or the index of $i(e)$ is $2n -2$, 
then $\phi(i(e)) < \phi(t(e))$.
\end{Corollary}

\begin{proof}
Let $e$ be an edge such that the index of $t(e)$ is $2$. By the lemma above,
the index of $i(e)$ in the isotropy submanifold $M^{\Z/(w(e))}$ is $0$. Hence, $i(e)$ is the
unique point in the component of  $M^{\Z/(w(e))}$ containing $t(e)$ 
where $\phi$ achieves its minimum. A fortiori, $\phi(i(e)) < \phi(t(e))$.
The remaining case is similar.
\end{proof}

To prove our main theorem, we need the following proposition:

\begin{Proposition}\labell{order}
Let the circle act on a $2n$-dimensional closed symplectic manifold $(M,\omega)$ with exactly $n+1$ fixed points and with moment map
$\phi \colon M \to \R$.
Then we can label the fixed points $p_0,p_1,\dots,p_n$ so that $p_i$ has
index $2i$, and
$$ \phi(p_i) < \phi(p_j) \quad \mbox{and}  \quad \Gamma_{p_i} > \Gamma_{p_j} \quad \mbox{for all } i<j .$$
\end{Proposition}

\begin{proof}
As we saw in the introduction,  
we can label the fixed points $p_0,p_1, \dots, p_n$ so that  $p_i$ has index $2i$;
moreover, $H^i(M;\R) = H^i(\CP^n;\R)$ for all $i$.
Therefore, the claim follows immediately from  Proposition 3.4 and Lemma 
3.23 in \cite{T}.
\end{proof}

We are now ready to prove our key technical result.  Recall from the introduction that, 
when taken together with Theorem~\ref{GStheorem}, it immediately implies Theorem~\ref{main}.

\begin{Theorem}\labell{positive}
Consider a Hamiltonian circle action on an $8$-dimensional closed symplectic manifold $(M,\omega)$ 
with exactly $5$ fixed points.
Then $M$ satisfies the positivity property $(\mathcal P^+_0)$.
\end{Theorem}

\begin{proof}
By Proposition~\ref{order}, we can label the fixed points $p_0,p_1,\dots,p_4$ so that $p_i$ has
index $2i$, and
$\phi(p_i) < \phi(p_j)$ and $\Gamma_{p_i} > \Gamma_{p_j}$  for all $ i<j.$
Hence, by Lemma~\ref{decreases} and Corollary~\ref{cor:decreases},
$M$ is described by a labelled multigraph with no self-loops such that
the edge $e'$ with $t(e') = p_1$ and the edge $e''$ with $i(e'') = p_3$ 
satisfy $\phi(i(e')) < \phi(t(e'))$ and $\phi(i(e'')) < \phi(t(e''))$.
Therefore, we must have $i(e') = p_0$
and $t(e'') = p_4$.
By considering the remaining istropy weights,
it is clear that for any other edge $e$, the index of $i(e)$ is at most $4$ and
the index of $t(e)$ is at least $4$. Since there is only one point of index four and no self-loops,
this implies that the index of $i(e)$ is less than the index of $t(e)$ for every edge $e$.
The claim follows immediately.
\end{proof}

\end{document}